\documentclass[12pt]{article}
\usepackage{epsfig}
\usepackage{amssymb,amsmath,amsthm,amscd}
\usepackage{latexsym}
\usepackage{xcolor,bm}
\usepackage{xcolor,bm}
\usepackage{cite}
\usepackage{color} 
 \theoremstyle{theorem}
 \newtheorem{thm}{Theorem}[section]
 \newtheorem{lem}{Lemma}[section]

 \newtheorem{defn}{Definition}[section]
 \newtheorem{ex}{Example}[section]

\newcommand{\cA}{{\mathcal A}}

\newcommand{\cM}{{\mathcal M}}

\newcommand{\cR}{{\mathcal R}}

\newcommand{\cV}{{\mathcal V}}

\def\Z{\mathbb{Z}}
\def\R{\mathbb{R}}

\def\1{\mathbb{1}}

\def\bc{\begin{center}}
\def\ec{\end{center}}
\def\be{\begin{equation}}
\def\ee{\end{equation}}
\def\ba{\begin{array}}
\def\ea{\end{array}}
\def\benu{\begin{enumerate}}
\def\eenu{\end{enumerate}}
\def\bt{\begin{theorem}}
\def\et{\end{theorem}}
\def\bl{\begin{lemma}}
\def\el{\end{lemma}}
\def\bco{\begin{corollary}}
\def\eco{\end{corollary}}
\def\br{\begin{remark}}
\def\er{\end{remark}}
\def\bd{\begin{definition}}
\def\ed{\end{definition}}
\def\bp{\begin{proposition}}
\def\ep{\end{proposition}}
\def\bo{\begin{proof}}
\def\eo{\end{proof}}
\def\bx{\begin{example}}
\def\ex{\end{example}}

\def\al{\alpha}

\def\lam{\lambda} 

\def\ve{\varepsilon}
\def\Sig{\Sigma}

\def\w{\omega}
\def\gam{\gamma}\def\leq{\leqslant}\def\geq{\geqslant}

\def\~{\widetilde}
\def\A{\forall}

\def\ra{\rightarrow}

\def\8{\infty}

\def\mb{\mbox}

\def\sm{\setminus}

\def\ss{\subset}
\def\emp{\emptyset}

\def\.{\cdot}

\def\Hs{\hspace{0.8cm}}
\def\hs{\hspace{0.4cm}}
\def\Vs{\vskip10pt}
\def\vs{\vskip5pt}

\def\[{\left[}
\def\]{\right]}
\def\({\left(}
\def\){\right)}

\title{A Remark on Attractor Bifurcation
  \footnote{This paper is supported by the grant of China
(11871368, 11801190)}}

\author{
\small{Chunqiu Li$^a$} \footnote{Corresponding author. Email: lichunqiu@wzu.edu.cn},\hs {Desheng Li$^b$
}\footnote{Email: lidsmath@tju.edu.cn},\hs {Jintao Wang$^c$} \footnote{Email: wangjt@hust.edu.cn}  ~\\\\
\small\it  $^{a}$Department of Mathematics, Wenzhou University, \\
\small\it Wenzhou, Zhejiang, 325035, P. R. China~\\
\small\it  $^b$School of Mathematics,  Tianjin University\\
\small\it Tianjin 300072, P. R. China~\\
\small\it  $^{c}$Center for Mathematical Sciences, Huazhong University of \\
\small\it Science and Technology,\,Wuhan 430074, P. R. China
}
\date{\small\today}

\begin{document}

\maketitle
\begin{abstract}\baselineskip 14pt
In this paper we present some local dynamic bifurcation results in
terms of invariant sets of nonlinear evolution equations. We show
that if the trivial solution is an isolated invariant set of the
system at the critical value $\lam=\lam_0$, then either there exists
a one-sided neighborhood $I^-$ of $\lambda_0$ such that for each
$\lambda\in I^-$, the system bifurcates from the trivial solution to
an isolated nonempty compact invariant set $K_\lambda$ with
$0\not\in K_\lambda$, or there is a one-sided neighborhood $I^+$ of
$\lam_0$ such that the system undergoes an attractor bifurcation for
$\lam\in I^+$ from $(0,\lam_0)$. Then we give a modified version of the attractor bifurcation theorem. Finally, we consider the classical Swift-Hohenberg equation and illustrate how to
apply our results to a concrete evolution equation.\\\\
\textbf{Keywords}: Invariant-set bifurcation; attractor bifurcation; Nonlinear evolution equation
 \\
\textbf{MSC2010}: 37B30, 37G35, 35B32
\end{abstract}


\setcounter {equation}{0}
\section{Introduction}
\hs\,\, This note  is concerned with the dynamic bifurcation of the nonlinear evolution equation
\be\label{e1.1}
u_t+Au=f_\lam(u)
\ee
 on a Banach space $X$, where $A: X^{\alpha}\rightarrow X$ is a
sectorial operator with compact resolvent for some
$0\leq\alpha<1$, $\lam\in\R$ is the bifurcation parameter, and
$f_{\lambda}(u)$ is a
locally Lipschitz continuous mapping from $X^{\alpha}\times \mathbb{R}$ to $X$. We also assume that
$f_\lam(0)=0$ for $\lam\in\R$ and
$$
  f_{\lambda}(u)=Df_{\lambda}(0)u+g_{\lambda}(u)
$$
with $Df_{\lambda}(u)$ continuous in $(u,\lambda)$, and that
$g_{\lambda}(u)=o(\|u\|_{X^{\alpha}})$ as
$\|u\|_{X^\alpha}\rightarrow 0$. So $u=0$ is always a trivial
solution of \eqref{e1.1} for each $\lambda\in \mathbb{R}$.

A quite fundamental result in the dynamic bifurcation theory is the
well-known Hopf bifurcation theorem \cite{H2,K,LWZH,MM}, which
concerns the bifurcation of a closed orbit from an equilibrium
point. The Hopf bifurcation theorem plays an important role in the
study of nonlinear dynamics, and has been fully developed in the
last century. However, it only applies to the case where there are
exactly a pair of conjugate eigenvalues of the linearized equation
crossing the imaginary axis. In practice, the linearized equation
near the equilibrium may have more than two eigenvalues  crossing
the imaginary axis. To deal with this case, a general dynamic
bifurcation theory needs to be developed, and this can be performed
in the context  of invariant sets \cite{BLL,CLR,CLR1,CH,LLZ,LW,R,W}, etc.


A particular but important case of  the invariant-set bifurcation is
the so-called attractor bifurcation, which was systematically
studied by Ma and Wang \cite{HMW,MW,MW1,MW2} and was further developed into a dynamic transition
theory \cite{MW3}. Sanjurjo
\cite{S} also addressed the attractor bifurcation theory from the
point of view of topology. Roughly speaking, the attractor
bifurcation theory states that   if the trivial solution  $0$ of
\eqref{e1.1} changes from an attractor to a repeller as  $\lambda$
crosses certain critical value $\lambda_0$, then the system
bifurcates  an attractor from the trivial solution. But in \cite{MW}
etc. it was assumed that the trivial solution is an attractor of the
system on the local center manifold when $\lambda=\lambda_0$.
Because the system is degenerate when it is restricted on the center
manifold, the verification of the condition that the trivial
solution is an attractor is often not an easy task.

In this paper we give a modified version of the
attractor bifurcation theorem in \cite{MW}, which drops the
additional condition mentioned above and makes the theorem more
efficient in applications. Specifically, let $\Phi_\lam$ be the
local semiflow generated by \eqref{e1.1}, and let $\cA_0=\{0\}$.
Suppose $\cA_0$ is an attractor of $\Phi_\lam$ for each
$\lam<\lam_0$, and that  there is at least one eigenvalue of the
linearized equation of (\ref{e1.1}) near the trivial solution
crossing the imaginary axis at the critical value
$\lambda=\lambda_0$. We prove that if $\mathcal{A}_0$ is an isolated
invariant set of $\Phi_{\lam_0}$, then there exists $\ve_1>0$ such
that the system $\Phi_\lambda$   bifurcates an isolated invariant
set $K_\lam$ with $0\not\in K_\lambda$ for each $\lam\in
[\lam_0-\varepsilon_1,\,\lam_0)$, or it bifurcates an attractor for
each  $\lam\in (\lam_0,\lam_0+\ve_1]$. In particular, if $\cA_0$ is
the global attractor of $\Phi_\lam$ for each $\lam<\lam_0$, then it
immediately follows that the system undergoes an attractor
bifurcation on $(\lam_0,\lam_0+\ve_1]$. Note that we do not assume
the trivial solution to be  an attractor of the system at
$\lambda=\lambda_0$.

This work is organized as follows. In Section 2 we introduce some
basic concepts and results concerning invariant sets. In Section 3
we first show our main results and then give the proofs of the
results. Finally we consider the classical Swift-Hohenberg equation
to illustrate our results.

\setcounter {equation}{0}
\section{Preliminaries}

\hs \,\, In this section we introduce some basic concepts
concerning local semiflows. First, let $X$ be a
complete metric space with metric $d(\cdot,\cdot)$.

\begin{defn}\label{defn2.1}
A local semiflow $\Phi$ defined on $X$ is a continuous map from an open
set $\mathcal{D}(\Phi)\subset \mathbb{R}^+\times X$ to $X$ and
satisfies the following properties:

{\rm(1)} For each $x\in X$, there exists
$0<T_x\leq\infty$ such that
$$(t,x)\in\mathcal{D}(\Phi)\Longleftrightarrow t\in [0,T_x).$$

{\rm(2)} $\Phi(0,\cdot)={\rm id}_X$, and
$$\Phi(t+s,x)=\Phi(t,\Phi(s,x))$$

for every $x\in X$ and $t,s \geq 0$ with $t+s\leq T_x$.
\vs
The number $T_x$ is called the escape time of $\Phi(t,x)$.
\end{defn}

For simplicity, we usually rewrite $\Phi(t,x)$ as $\Phi(t)x$.

Let $I\subset\R$ be an interval. A {\it trajectory} (or {\em solution}) of $\Phi$ on $I$
is a continuous mapping $\gam:I\ra X$ with
$$
  \gam(t)=\Phi(t-s)\gam(s),\hs \forall t,s\in I,\, t\geq s.
  $$
If $I=\R$, the trajectory $\gam$ is called  a {\it full trajectory}.

The {\em $\omega$-limit set} $\w(\gam)$ and {\em $\al$-limit set} $\al(\gam)$ of a full trajectory $\gam$ are defined as
$$\ba{ll}
\w(\gamma)=\{y\in X:\,\,\,\,\mb{there exists }  t_n\ra \8 \mb{ such that }\gamma(t_n)\ra y\},\ea$$
$$
\ba{ll}
\al(\gamma)=\{y\in X:\,\,\,\,\mb{there exists }  t_n\ra -\8 \mb{ such that }\gamma(t_n)\ra y\}.\ea
$$
\begin{defn}\label{defn2.2}
Let $N\subset X$. We say that $\Phi$ does not explode in $N$, if we
can infer $T_x=\infty$ from $\Phi([0,T_x))x\subset N$.

\end{defn}

\begin{defn}\label{defn2.3}\cite{R}
$N\subset X$ is said to be admissible, if for every sequences $x_n\in
N$ and $t_n\rightarrow\infty$ with $\Phi([0,t_n])x_n\subset N$ for
all $n$, the sequence $\Phi(t_n)x_n$ has a convergent subsequence.

$N$ is said to be strongly admissible, if it is admissible and
moreover, $\Phi$ does not explode in $N$.
\end{defn}

\begin{defn}\label{defn2.4}
$\Phi$ is said to be asymptotically compact on $X$, if each bounded
subset $B$ of $X$ is strongly admissible.
\end{defn}

\vs
From now on, we always assume that

\vs
$(\mathbf{AC})$ $\Phi$ is asymptotically compact on $X$.

\begin{defn}\label{defn2.5}
A set $\mathcal{A}\subset X$ is said to be positively invariant
(resp., invariant) for $\Phi$, if
$\Phi(t)\mathcal{A}\subset\mathcal{A}$\,(resp.,
$\Phi(t)\mathcal{A}=\mathcal{A}$) for each $t\geq 0$. An invariant
set $\mathcal{A}$ is called an attractor of $\Phi$, if it is compact
and attracts a neighborhood $U$ of itself, that is,
$$
  \lim\limits_{t\rightarrow\infty}\mathrm{dist}_H(\Phi(t)U,\mathcal{A})=0.
  $$
The attraction basin of $\mathcal{A}$, denoted by
$\Omega(\mathcal{A})$, is defined as
$$
  \Omega(\mathcal{A})=\{x\in X:
  \lim\limits_{t\rightarrow\infty}\mathrm{dist}_H(\Phi(t)x,\mathcal{A})=0\},
$$
where $\mathrm{dist}_H$ denotes the Hausdorff semi-distance. %
\end{defn}

Suppose $M$ is a compact invariant set. Then the restriction
$\Phi_M$ of $\Phi$ on $M$ is also a semiflow. A compact set
$\mathcal{A}\subset M$ is called an {\it attractor} of $\Phi$ in
$M$, which means that $\mathcal{A}$ is an attractor of $\Phi_M$ in
$M$.

Let $\mathcal{A}$ be an attractor of $\Phi$ in $M$. Define
$$
  \mathcal{R}=\{ x\in M: \omega(x)\cap \mathcal{A}=\emptyset\}.
  $$
Then $\mathcal{R}$ is called the {\it repeller} of $\Phi$ in $M$
dual to $\mathcal{A}$, and $(\mathcal{A},\mathcal{R})$ is called an
{\it attractor-repeller pair} of $\Phi$ in $M$ .

\begin{lem}\label{l2.1}
Let $\mathcal{R}\subset M$ be a nonempty compact invariant set.
Suppose that there exists an open neighborhood $W$ of $\mathcal{R}$
in $M$ such that for each $x\in W\setminus \mathcal{R}$ and each
complete trajectory $\gamma$ in $M$ through $x$, one has
$$
  \alpha(\gamma)\subset \mathcal{R}, \hs\omega(\gamma)\cap W=\emptyset.
  $$
Denote $\mathcal{A}$ the maximal compact invariant set in
$M\setminus W$. Then $(\mathcal{A},\mathcal{R})$ is an
attractor-repeller pair.
\end{lem}
\bo
By the definition of $\cA$, it is clear that $\cA$ and $\cR$ are disjointed. Let $x\in M\sm(\cA\cup \cR)$, and let $\gam$ be a complete trajectory in $M$ through $x$. We claim that there exists a $t_0\in \R$ such that $\gam(t_0)\in W$. Indeed, if this was not the case, then one would have $\gam(t)\ss M\sm W$ for all $t\in\R$. Therefore by the definition of $\cA$, we find that $\gam$ is contained in $\cA$. This leads to a contradiction (as $\gam(0)=x\notin\cA$).

Now by the assumption of the lemma, one easily verifies that $$\alpha(\gam)\ss\cR,\hs \w(\gam)\ss \cA,$$ and the conclusion of the lemma follows immediately from Theorems 1.7 and 1.8 on Morse decompositions of invariant sets in \cite{R}, Chap. III.
 \eo
Next, we recall some basic concepts and results concerning the
Conley index theory. One can refer to \cite{C,Mis,R}, etc. for details.

\vs
Let $N,E$ be two closed subsets of $X$. We say that $E$ is an {\it exit
set} of $N$, if it satisfies
\vs
{\rm(1)} $E$ is N-positively invariant, that is, if for any $x\in E$
and $t\geq 0$,
$$\Phi([0,t])x\subset N \Longrightarrow\Phi([0,t])x\subset E;$$

{\rm(2)} For any $x\in N,$ if $\Phi(t_1)x\not\in N$ for some
$t_1>0$, then there exists a $t_0\in[0,t_1]$ such that $\Phi(t_0)x\in
E$.
\vs

A compact invariant set $\mathcal{A}$ of $\Phi$ is said to be
{\it isolated}, if there exists a bounded closed neighborhood $N$ of
$\mathcal{A}$ such that $\mathcal{A}$ is the maximal invariant set
in $N$. Consequently, $N$ is called an {\it isolating neighborhood} of
$\mathcal{A}$.

Let $\mathcal{A}$ be a compact isolated invariant set. A pair of
closed subsets $(N,E)$ is said to be an {\it index pair} of $\mathcal{A}$,
if it satisfies the following conditions:
\vs
{\rm(1)} $N\setminus E$ is an isolating neighborhood of
$\mathcal{A};$
\vs
{\rm(2)} $E$ is an exit set of $N$.

\begin{defn}\label{defn2.8}
Let $(N,E)$ be an index pair of $\mathcal{A}$. Then the homotopy
Conley index of $\mathcal{A}$ is defined to be the homotopy
type $[(N/E,[E])]$ of the pointed space $(N/E,[E])$, denoted by
$h(\Phi,\mathcal{A}).$
\end{defn}


Next we present an important result on the continuation property of
Conley index, which plays an important role in the proof of our
invariant sets bifurcation results. 

Let $\Phi_\lambda$ be a family of semiflows with parameter
$\lambda\in\Lambda$, where $\Lambda$ is a connected compact metric
space. Suppose that $\Phi_\lambda(t)x$ is continuous in
$(t,x,\lambda)$. Define the {\it skew-product flow}
$\widetilde{\Phi}$ of the family $\Phi_\lambda$ on $X\times\Lambda$
by
$$
  \widetilde{\Phi}(t)(x,\lambda)=(\Phi_\lambda(t)x,\lambda), \Hs (x,\lambda)\in
  X\times\Lambda.
  $$

\begin{lem}\label{thm 2.1}\cite{LW}
Let $\widetilde{\Phi}$ satisfy the assumption $(\mathbf{AC})$ on
$X\times\Lambda$. Suppose $\mathcal{M}$ is a compact isolated
invariant set of $\widetilde{\Phi}$. Then
$h(\Phi_\lambda,M_\lambda)$ is constant for $\lambda\in \Lambda$,
where $M_\lambda=\{x:(x,\lambda)\in \mathcal{M}\}$ is the
$\lambda$-section of $\mathcal{M}.$

\end{lem}

Suppose that $B\subset X$ is a bounded closed set. $x\in\partial B$
is said to be a {\it strict egress} (resp., {\it strict ingress}, {\it bounce-off})
{\it point} of $B$, if for each trajectory
$\gamma:[-\delta,\tau]\rightarrow X$ with $\gam(0)=x$, where $\delta\geq 0,\tau>0$, the
following two properties hold.

(1) There exists $0<s<\tau$ so that
$$
  \gamma(t)\not\in B \hs({\rm resp}.,\hs\gamma(t)\in {\rm int}B, \hs{\rm resp}., \hs\gamma(t)\not\in B),\Hs \forall t\in(0,s);
$$

(2) If $\delta>0$, then there exists $0<\beta<\delta$ such that
$$
  \gamma(t)\in {\rm int} B \hs({\rm resp}., \hs\gamma(t)\not\in B, \hs{\rm resp}., \hs\gamma(t)\not\in B),\Hs \forall
  t\in(-\beta,0).
$$

Denote  the set of all strict egress (resp. strict ingress,
bounce-off) points of the closed set $B$ by $B^e(\,{\rm resp}.\,
B^i,\,\, B^b)$, and set $B^-=B^e\cup B^b$. For convenience in
statement, if $B^-$ is the exit set, we call $B^-$ the {\it boundary
exit set} of $B$.
\vs

A closed set $B\subset X$ is called an {\it isolated block}\cite{R} if $B^-$ is
closed and $\partial B=B^i\cup B^-$.

Suppose that $K\subset X$ is a compact isolated invariant set and
the isolating block $B$ is the isolating neighborhood of $K$. If
$B^-$ is B-positively invariant, then $B$ is called the {\it index
neighborhood} of $K$.

By the definition of index neighborhood, we have the following
result.
\begin{thm}\label{L2.1}
Let $K\subset X$ be a compact isolated invariant set and $N$ be the
isolating neighborhood of $K$. Then there exists an isolating block
$B$ in $N$ such that $B$ is the index neighborhood of $K$.
\begin{proof}
By Chapter 1, Theorem 5.1 in \cite{R}, we deduce that there exists a
bounded closed set $B\subset N$ with $K\subset B$ such that $B$ is
an isolating block. Moreover, from \cite{R}, it holds that if $B$ is
a bounded isolating block, then $(B,B^-)$ is an index pair of the
maximal compact invariant $K$ in $B$. Thus the result follows from
the definition of index neighborhood.
\end{proof}
\end{thm}

\setcounter {equation}{0}
\section{Invariant-set/attractor bifurcation}

\hs \,\, In this section, we establish some local bifurcation results in terms
of invariant sets.
\subsection{Main results}

\hs \,\, It is well known that (see e.g. \cite{H,SY}) that the Cauchy problem of \eqref{e1.1} is
well-posed in $X^\alpha$ under the assumptions in Section 1.
Specifically, for any initial value $u_0\in X^\alpha$, the equation
\eqref{e1.1} has a unique continuous solution $u(t)\in X^\alpha$
with $u(0)=u_0$ on a maximal existence interval $[0,T)$ for some $T>0$. Let $\Phi_\lambda$ be the local semiflow
generated by the equation \eqref{e1.1} on $X^\alpha.$

\vs

Let $L_{\lambda}=A-Df_{\lambda}(0)$. Assume there exist a
neighborhood $J_0=(\lambda_0-\varepsilon, \lambda_0+\varepsilon)$ of
$\lambda_0$ and a positive constant $\beta$ such
that the following hypotheses $(\mathbf{A}1)$-$(\mathbf{A}3)$ are satisfied.

 \vs

$(\mathbf{A}1)$ The spectrum $\sigma(L_\lambda)$ has a decomposition
 $\sigma(L_\lambda)=\bigcup_{1\leq i\leq 2}\sigma^{i}_{\lambda}$ such that
\begin{align*}
  \max_{\mu\in\sigma_{\lambda}^1}|\mathbf{Re}(\mu)|\leq \beta \hs{\rm
and}\hs
  \min_{\mu\in\sigma_{\lambda}^2}\mathbf{Re}(\mu)\geq 2\beta,
\hs \forall\lambda\in J_0.
\end{align*}

$(\mathbf{A}2)$ For every $\lambda\in J_0$, $X^\al$ has a decomposition
$X^\al=X_{\lambda}^1\oplus X_\lambda^2$ with respect to the spectral
decomposition in $(\mathbf{A}1)$, where $X_\lambda^i(i=1,2)$ are
$L_\lambda$-invariant subspaces of $X^\al$. Moreover,
$$
  1\leq\dim(X_\lambda^1)<\infty.
$$

$(\mathbf{A}3)$  The projection operator $P_\lam^{1}:X^\al\ra X^1_\lam$ is continuous in $\lam$.
\Vs

For  simplicity,  from now on we drop the subscript ``$\lam_0$'' and
rewrite $X^i=X_{\lam_0}^i.$ Let $E=X^\alpha$, and $$E^i=E\cap
X^i,\Hs i=1,2.$$ Then $E=E^1\oplus E^2$.  Furthermore, since $X^1$
is finite dimensional, we have $E^1=X^1$ and $n:=\dim(E^1)\geq
1$.

By virtue of  Proposition 2 in Appendix A, there is a family of
isomorphisms $T=T_\lam$ ($\lam\in J_0$) on $X$ depending
continuously on $\lam$ with $T_{\lam_0}=I$, such that
$$
    T X^i_\lam=X^i_{\lam_0}:=X^i,\Hs i=1,2.$$

Next we introduce the definition of invariant sets bifurcation, and
then we state and prove the local invariant sets bifurcation result.

\begin{defn}\label{defn3.1}Let $I$ be a subset of $\R$. We say that
the equation \eqref{e1.1} undergoes an invariant-set bifurcation on
$I$ from the trivial solution $(0,\lambda_0)$, if there is a
sequence $\lam_n\in I$, $\lam_n\ra\lam_0$ as $n\ra\8$, such that
 $\Phi_{\lambda_n}$ has an invariant set $\mathcal{A}_{\lambda_n}$ with $\cA_{\lam_n}\setminus\{0\}\ne\emp$; furthermore,
$$
  \lim\limits_{n\rightarrow\8}\mathrm{dist}_H(\mathcal{A}_{\lambda_n},0)=0.
  $$

If each invariant set $\mathcal{A}_{\lambda_n}$ is an attractor of
the system with $0\not\in \cA_{\lam_n}$, then we say that
\eqref{e1.1} undergoes an attractor bifurcation on $I$ from
$(0,\lambda_0)$.
\end{defn}

\begin{thm}\label{L3.0}
Let the assumptions $(\mathbf{A}1)$-$(\mathbf{A}3)$ hold true.
Suppose
\begin{align*}
 (\mathbf{A}4) \hs
 \min_{\mu\in\sigma_{\lambda}^1}\mathbf{Re}(\mu)>0 \,\,(\lam  \in(\lam_0-\varepsilon,\lam_0)),\hs
 \max_{\mu\in\sigma_{\lambda}^1}\mathbf{Re}(\mu)<0\,\,(   \lambda \in(\lam_0,\lam_0+\varepsilon)).
\end{align*} If $\mathcal{A}_0:=\{0\}$ is an isolated invariant set of $\Phi_{\lam_0}$,
then there exist $\varepsilon_1>0$ and a closed neighborhood $W$ of
$0$ in $E$ such that one of the alternatives holds.
  \benu
  \item[(1)]

  The system undergoes  an invariant-set bifurcation on $I^-=[\lam_0-\varepsilon_1,\,\lam_0)$ from $(0,\lam_0)$. More precisely,
  for any $\lam\in I^-$, $\Phi_\lam$ has a nonempty  compact invariant set $K_\lam$ with $K_\lam$ in $W\setminus\{0\}$ such that
\be\label{e3.2}
\lim\limits_{\lam\rightarrow\lam_0}\mathrm{dist}_H(K_{\lambda},0)=0.
\ee

  \item[(2)] The system undergoes  an attractor bifurcation on $I^+=(\lam_0,\lam_0+\varepsilon_1]$ from $(0,\lam_0)$.
  Specifically, for any $\lam\in I^+$, $\Phi_\lam$ has an attractor $K_\lam$ with $K_\lam$ in $W\setminus\{0\}$ such that \eqref{e3.2} holds.
  Furthermore, $K_\lam$ contains an invariant topological sphere
  $\mathbb{S}^{n-1}$.
  \eenu
\end{thm}

As a simple consequence of the above theorem, we have the following attractor bifurcation result.

\begin{thm}\label{AT}
Let the assumptions $(\mathbf{A}1)$-$(\mathbf{A}4)$ hold true.
Suppose also that $\mathcal{A}_0:=\{0\}$ is a global attractor for each $\lam<\lam_0$.

Then either $\cA_0$ is not isolated with respect to $\Phi_{\lam_0}$,
or the system undergoes an attractor bifurcation on
$I^+=(\lam_0,\lam_0+\varepsilon_1]$ from $(0,\lam_0)$ for some
$\varepsilon_1>0$. Specifically, for any $\lam\in I^+$, $\Phi_\lam$
has an attractor $K_\lam$ with $K_\lam$ in $W\setminus\{0\}$ for some closed neighborhood $W$ of $0$ in $E$ such
that \eqref{e3.2} holds and $K_\lam$ contains an invariant
topological sphere $\mathbb{S}^{n-1}$.
 \end{thm}

\vs \noindent \textbf{Remark}. The above result drops an additional
assumption that the trivial solution $\cA_0$ is an attractor on the
local center manifold with respect to the system $\Phi_{\lam_0}$,
just as what is expressed in Ma and Wang \cite{MW} and only assumes
that the trivial solution $\cA_0$ is an isolated invariant set of
$\Phi_{\lam_0}$. In this sense, our result is more general and
extends the attractor bifurcation result obtained in \cite{MW}.

\subsection{Proof of the main results}

\hs \, In this subsection we will give the proof of our main result. For
this purpose, we first introduce some lemmas, which are crucial in
our proof.

\begin{lem}\label{L3.1}
Let the assumptions $(\mathbf{A}1)$-$(\mathbf{A}3)$ hold true. Then
there exist an open convex neighborhood $U$ of $0$ in $E^1$ and a
mapping $\xi=\xi_\lambda(v)$ from  $U\times J_0$ to $E^2$, which is
continuous in $(v, \lambda)$ and differentiable in $v$, such that
for any $\lambda\in J_0$, \be\label{e3.3}
  M_\lambda:=T^{-1}\mathcal{M}_\lam, \Hs where \hs \mathcal{M}_\lam:= \{v+\xi_\lambda(v), \hs v\in U\},
\ee is a local invariant manifold of the system \eqref{e1.1} with
$\xi_\lambda(0)=D\xi_\lambda(0)=0$.
\end{lem}

The proof of Lemma 3.1 follows from the standard argument in the
geometric theory of PDEs and the uniform contraction principle, one
can see Henry \cite{H} and Hale \cite{H1} for details.

\begin{lem}\label{L3.2}\cite{LW}
Let the assumptions $(\mathbf{A}1)$-$(\mathbf{A}3)$ hold true and
$M_\lambda$ be the local invariant manifold obtained in Lemma 3.1,
and $\Phi_\lambda^1$ be the restriction of $\Phi_\lambda$ on
$M_\lambda.$

Then there exist a neighborhood $U_0$ of $0$ and $\varepsilon_0>0$
such that for each
$\lambda\in[\lambda_0-\varepsilon_0,\lambda_0+\varepsilon_0],$
$\mathcal{A}_\lambda$ is an isolated invariant set of $\Phi_\lambda$
in $U_0$ iff it is an isolated invariant set of $\Phi_\lambda^1$ on
$M_\lambda.$ Furthermore,
$$
  h(\Phi_\lambda,\mathcal{A}_\lambda)=h(\Phi_\lambda^1,\mathcal{A}_\lambda).
  $$

\end{lem}

The proof of Theorem 3.1 is as follows.

\begin{proof}

Because the trivial solution $\cA_0$ is isolated with respect to
$\Phi_{\lam_0}$, two cases my occur.

\vs

\textbf{Case one.} If $\cA_0$ is not an attractor, next we prove
that (1) of Theorem 3.1 holds true.

\vs

According to assumption $(\mathbf{A}4)$, we see that the equilibrium
$\mathcal{A}_0=\{0\}$ is an attractor of $\Phi_\lambda$ for $\lambda
\in(\lam_0-\varepsilon,\lam_0).$

Set $B_\lam=TL_\lam T^{-1}$ and define
$$
  g_\lam(v)=T(f_\lam(T^{-1}v)-Df_\lam(0)(T^{-1}v)),\Hs v\in E.
$$
Then the system \eqref{e1.1} can be transformed into the following
equivalent equation by letting $u=T^{-1}v$,
\begin{align}\label{e3.4}
v_t=-B_\lambda v+g_\lambda(v),
\end{align}
for $\lambda\in J_0$. When system \eqref{e3.4} is restricted on
the local center manifold $\mathcal{M}_\lambda$ defined by
\eqref{e3.3}, it reduces to an ODE system on a neighborhood $U$
(independent of $\lambda$) of $\mathcal{A}_0$ in $E^1$:
\begin{align}\label{e3.5}
v_t=-B_\lambda^1 v+P^1g_\lambda(v+\xi_\lambda(v)):=G_\lambda(v),
\end{align}
where $B_\lambda^1=P^1B_\lambda,$ and $P^1:E=X^\alpha\rightarrow
E^1$ is the projection operator. Applying Lemma 3.2 to system
\eqref{e3.4} , we conclude that there exists a neighborhood
$\mathcal{U}$ of $0$ in $E$ and $\varepsilon_0>0$ such that for any
$\lambda\in[\lam_0-\varepsilon_0,\lam_0+\varepsilon_0]$,
$\mathcal{A}_\lambda$ is an isolated invariant set of system
\eqref{e3.4} in $\mathcal{U}$ iff it is an isolated invariant set of
the system restricted on the manifold $\mathcal{M}_\lambda.$

Since $\mathcal{A}_0$ is isolated, we can
choose a bounded closed neighborhood $W$ such that $W$ is an
isolated neighborhood of $\mathcal{A}_0$ with respect to
$\Phi_{\lam_0}$. By a simple argument via contradiction, one can
easily verify that $W$ is also an isolating neighborhood of the
maximal compact invariant set $\mathcal{A}_\lambda$ of
$\Phi_\lambda$ in $W$, provided $\lambda $ near $\lam_0$. That is to
say, there exists $0<\varepsilon_1\leq\varepsilon_0$ so that $W$ is
also an isolated neighborhood of $\mathcal{A}_\lambda$ for every
$\lambda\in[\lam_0-\varepsilon_1,\lam_0+\varepsilon_1]$. It is
trivial to check that
\begin{align}\label{e3.6}
  \lim\limits_{\lambda\rightarrow
  \lam_0}\mathrm{dist}_H(\mathcal{A}_{\lambda},\mathcal{A}_0)=0.
\end{align}
Moreover, Lemma 3.2 also shows that $\mathcal{A}_\lambda$ is also an
isolated invariant set of the system on the local center
manifold $\cM_\lambda$ for
$\lambda\in[\lam_0-\varepsilon_1,\lam_0+\varepsilon_1]$. Denote by
$\phi_\lambda$ the semiflow generated by the system \eqref{e3.5}.
Because the topological structure of the solutions of system
\eqref{e3.4} on $\mathcal{M}_\lambda$ is equivalent to that of the
system \eqref{e3.5} on $U$, $\mathcal{A}_0$ is also an isolated
invariant set of $\phi_0$ on $U$. Then one can pick a closed isolated
neighborhood $U_0$ of $\mathcal{A}_0$ satisfying $U_0\subset U$.
By Theorem 2.1, we take a closed neighborhood
$N_0\subset U_0 $ of $\mathcal{A}_0$ such that $N_0$ is an index
neighborhood of $\mathcal{A}_0$ with respect to $\phi_{0}$. We may
suppose that $N_0$ is path-connected. By virtue of \eqref{e3.6}, it
can be assumed that the $N_0$ is also an isolating neighborhood of
$\mathcal{A}_\lambda$ of $\phi_\lambda$ for
$\lambda\in[\lam_0-\varepsilon_1,\lam_0+\varepsilon_1]$. Therefore by
the continuation property of Conley index, we have
\begin{align}\label{e3.7}
  h(\mathcal{A}_\lambda,\phi_\lambda)\equiv const,\Hs \lambda \in
  [\lam_0-\varepsilon_1,\lam_0+\varepsilon_1].
\end{align}

On the other hand, because $\mathcal{A}_0$ is an attractor of
$\Phi_{\lambda}$ for $\lambda\in(\lam_0-\varepsilon,\lam_0)$, one
deduces that $\mathcal{A}_0$ is also an attractor of $\phi_\lambda$
on $U$ for $\lambda\in[\lam_0-\varepsilon_1,\lam_0)$. So there also
exists a path-connected closed neighborhood $N_1$ of $\mathcal{A}_0$
such that $N_1$ is an isolated neighborhood of $\mathcal{A}_0$ with
respect to $\phi_\lambda$ for
$\lambda\in[\lam_0-\varepsilon_1,\lam_0)$. Note that $\mathcal{A}_0$
is not an attractor for $\phi_{0}$, one concludes that the boundary exit set
$N_0^-\neq\emptyset$. By some elementary computations of homology Conley
index of $\mathcal{A}_0$, we have
\begin{align}\label{e3.8}
&
 H_0(h(\mathcal{A}_0,\phi_{0}))=H_0([N_0/N_0^-,[N_0^-]])=0,\nonumber\\
&
 H_0(h(\mathcal{A}_0,\phi_{\lambda}))=H_0([N_1/\emptyset,[\emptyset]])=H_0([N_1\cup\{p\}/\{p\},\{p\}])=\Z,
\end{align}
for $\lambda\in[\lam_0-\varepsilon_1,\lam_0)$, where $p\notin N_1$, which implies $h(\mathcal{A}_0,\phi_{0})\neq
h(\mathcal{A}_0,\phi_{\lambda})$. Therefore, we deduce from \eqref{e3.7} that for each
$\lambda\in[\lam_0-\varepsilon_1,\lam_0)$,
$$
  h(\mathcal{A}_\lambda,\phi_\lambda)=h(\mathcal{A}_0,\phi_{0})\neq h(\mathcal{A}_0,\phi_{\lambda}).
  $$
Then it follows that $\mathcal{A}_\lambda\setminus
\mathcal{A}_0\neq\emptyset.$ Since $\mathcal{A}_0$ is an attractor
of $\phi_\lambda$, we conclude that the set $$K_\lam:=\{x\in \cA_\lam:\w(x)\cap\cA_0=\emp\}$$ is a nonempty compact invariant set of $\phi_\lam$ with $(\mathcal{A}_0, K_\lambda)$ being an attractor-repeller
pair of $\mathcal{A}_\lambda$, see \cite{R}, pp.141. Note that $\cA_\lam$ is maximal in $N_0$. So one can see that
$K_\lambda$ is also the maximal compact invariant set of $\phi_\lambda$
in $N_0\setminus \mathcal{A}_0$ for
$\lambda\in[\lam_0-\varepsilon_1,\lam_0).$
Thus, let $I^-=[\lam_0-\varepsilon_1,\lam_0)$ and by \eqref{e3.6},
we see that (1) holds true.

\vs

\textbf{Case two.} If $\cA_0$ is an attractor, then we prove that
(2) of Theorem 3.1 holds true. Indeed, the proof of this case is a
slight modification of the one for the corresponding result in Ma
and Wang \cite{MW}. Here we give the details for completeness and
the reader' convenience.

\vs

Similarly applying Lemma 3.2 to \eqref{e3.4}, we can take a
neighborhood $\mathcal{U}$ of $0$ in $E$ and
$\varepsilon_0\in(0,\varepsilon)$ such that for any
$\lambda\in[\lam_0-\varepsilon_0,\lam_0+\varepsilon_0]$,
$\mathcal{A}_\lambda$ is an isolated invariant set of $\Phi_\lambda$
in $\mathcal{U}$ if and only if it is an isolated invariant set of
the system on the manifold $\mathcal{M}_\lambda.$ Now we consider
the system \eqref{e3.5} on $U$, which is restricted on the local
center manifold $\mathcal{M}_\lambda$.

Let $\phi_\lambda$ be the semiflow generated by system \eqref{e3.5}
on $U$. Then one deduces that $\mathcal{A}_0$ is also an attractor
of $\phi_0$. Denote by $\Omega=\Omega(\mathcal{A}_0)$ the attraction
basin of $\mathcal{A}_0$ in $U$ with respect to $\phi_0$. Owing to
the converse Lyapunov theorems on attractors \cite{L}, we deduce
that there is a function $V\in C^\infty(\Omega)$ such that
\begin{align}\label{e3.9}
  \nabla V(x)\cdot G_{\lam_0}(x)\leq -w(x),\Hs \forall x\in\Omega,
\end{align}
and satisfies $V(0)=0$ and
$\lim\limits_{x\rightarrow\partial\Omega}V(x)=+\infty$, where $w\in
C(\Omega)$ and $w(x)>0$ for $x\neq 0$, $G_\lam$ is given by \eqref{e3.5}. Set
$$
  \mathcal{V}=V_a:=\{x\in\Omega: V(x)\leq a\}.
$$
Then one can easily conclude that $\mathcal{V}$ is a compact positively
invariant neighborhood of $0$ in $E^1$ for each $a>0$. Now we choose
two positive numbers $a,b$ sufficiently small such that
$$
  \tilde{W}:=\mathcal{V}\times B_{E^2} (\xi_{\lam_0}(\mathcal{V}),b)\subset \mathcal{U},
  $$
where $\xi_{\lam_0}$ is the local center manifold mapping obtained in
Lemma 3.1, and $B_{E^2} (\xi_{\lam_0}(\mathcal{V}),b)$ means the
$b$-neighborhood of $\xi_0(\mathcal{V})$ in $E^2$.

Let $c=\min\limits_{x\in\partial \mathcal{V}}w(x)>0.$ Thanks to
\eqref{e3.9}, we have
\begin{align}\label{e3.10}
\nabla V(x)\cdot G_{\lam_0}(x)\leq -c, \Hs \forall x \in\partial
\mathcal{V},
\end{align}
where $\partial \mathcal{V}$ denotes the boundary of $\mathcal{V}$
in $E^1.$ Moreover, by the continuity of $G_\lambda$, there exists
$\varepsilon_1\in(0,\varepsilon_0]$ such that for each
$\lambda\in[\lam_0,\lam_0+\varepsilon_1]$,
\begin{align}\label{e3.11}
\nabla V(x)\cdot G_\lambda(x)\leq -\frac{c}{2}, \Hs \forall x
\in\partial \cV.
\end{align}
Note that $\xi_\lambda$ is continuous in $\lambda$. Then it can be
assumed that the $\varepsilon_1$ is sufficiently small such that
$$
  \xi_\lambda(\mathcal{V})\subset B_{E^2} (\xi_{\lam_0}(\mathcal{V}),b),\Hs
  \lambda\in[\lam_0,\lam_0+\varepsilon_1].
$$
Therefore
\be\label{e3.12}
 \mathcal{V}\times \xi_\lambda(\mathcal{V})\subset \tilde{W}\subset \mathcal{U},\Hs
  \lambda\in[\lam_0,\lam_0+\varepsilon_1].
\ee
From \eqref{e3.11}, one can deduce that $\mathcal{V}$ is an
absorbing set of $\phi_\lambda$. So $\phi_\lambda$ has an attractor
$\mathcal{A}_\lambda$ which is the maximal invariant set of
$\phi_\lambda$ in $\mathcal{V}$ for $\lambda\in[\lam_0,\lam_0+\varepsilon_1]$.
By the upper semicontinuity of attractors \cite{MW}, we have
$$
  \lim\limits_{\lambda\rightarrow
  \lam_0}\mathrm{dist}_H(\mathcal{A}_{\lambda},0)=0.
$$
Recalling $\mathbf{Re}\sigma_{\lambda}^1<0$ for
$\lambda\in(\lam_0,\lam_0+\varepsilon_1]$, we see that
$\mathcal{A}_0:=\{0\}$ is a repeller of $\phi_\lambda$. Thus we conclude from
Lemma \ref{l2.1} that $\mathcal{A}_\lambda$ has an
attractor-repeller pair $(\mathcal{K}_\lambda, \mathcal{A}_0)$ for
$\lambda\in(\lam_0,\lam_0+\varepsilon_1]$, where
$\mathcal{K}_\lambda$ is the maximal compact invariant set of
$\phi_\lambda$ in $\mathcal{A}_\lambda\setminus\{0\}.$ By the
maximality of $\mathcal{A}_\lambda$, we conclude that
$\mathcal{K}_\lambda$ is also the maximal compact invariant set of
$\phi_\lambda$ in $\mathcal{V}\setminus\{0\}$ for
$\lambda\in(\lam_0,\lam_0+\varepsilon_1]$.

Next we verify that
$\mathcal{K}_\lam$ has an invariant $(n-1)$-dimensional topological
sphere. To this end, we consider the inverse flow $\phi_\lambda^-$
of $\phi_\lambda$ on $U$ generated by the following system
\begin{align*}
w_t=-G_\lambda(w).
\end{align*}
Thus, we see that $\mathcal{A}_0$ becomes an attractor of
$\phi_\lambda^-$ for $\lambda \in(\lam_0,\lam_0+\varepsilon_1]$. Let
$\Sigma:=\Sigma(\cA_0)$ denote the attraction basin of $\cA_0$ with
respect to $\phi^-_\lam$. It is trivial to check that
$\partial\Sigma$ is invariant under $\phi_\lam^-$. Therefore
$\partial\Sigma$ is also an invariant set of $\phi_\lam$, which
implies $\partial\Sigma\subset \mathcal{K}_\lam$. Now we claim
$\Sigma$ is contractible, so $\partial\Sigma$ is an
$(n-1)$-dimensional topological sphere. Indeed, define
\begin{align*}
H(s,x)=\left\{\ba{ll} \phi_\lam^-(\frac{s}{1-s})x, \Hs \,\, s\in
[0,1), x
\in\Sigma; \\[1ex]
0, \Hs \Hs \Hs s=1,x \in\Sigma.\ea\right.
\end{align*}
Hence $H$ is a strong deformation retract which shrinks $\Sigma$ to
$0$. Finally, we define
$$
  \tilde{\mathcal{K}}_\lam=\{v+\xi_\lam(v):v\in \mathcal{K}_\lam\}, \Hs
  \tilde{S}=\{v+\xi_\lam(v):v\in \partial\Sigma\}.
  $$
Since $\mathcal{K}_\lam\subset \mathcal{V}$, and from \eqref{e3.12}
we infer that $\tilde{\mathcal{K}}_\lam\subset \tilde{W} \subset
\mathcal{U}$. One can easily see that $\tilde{\mathcal{K}}_\lam$ is
the maximal compact invariant set of  \eqref{e3.4} in $\tilde{W}\setminus\{0\}$.
Thus let
$$
  W=T^{-1}\tilde{W},\hs K_\lam=T^{-1}\tilde{\mathcal{K}}_\lam,\hs
  \mathbb{S}^{n-1}=T^{-1}\tilde{S}.
$$
Then $W,K_\lam,\mathbb{S}^{n-1}$ satisfy the
requirements of (2) in Theorem 3.1. The proof of Theorem 3.1 is complete.


\end{proof}

\setcounter {equation}{0}
\section{Example}

\hs\,\, In this section, we give an example to illustrate how to apply our abstract results to a concrete evolution equation.

\vs

Consider the initial value problem of the classical Swift-Hohenberg
equation as follows:
\begin{align}\label{e4.1}
\left\{\ba{ll} u_t=-(I+\Delta)^2 u+\lambda u- u^3 , \hs(x,t)\in
\Omega\times \mathbb{R}^+,\\[1ex]
u(0)=u_0, \ea\right.
\end{align}
where $u:\Omega\times \mathbb{R}^+\rightarrow \mathbb{R}$ is a
real-valued function, $\Omega=(0,\pi)\subset \mathbb{R}$, and
$\lambda\in \mathbb{R}$ is the bifurcation parameter.

\vs

\noindent $\textbf{Remark 5.1}$ Concerning this Swift-Hohenberg
equation, in fact, the authors in \cite{Y} has obtained an attractor
bifurcation result of system \eqref{e4.1} by giving some precise
estimates of solutions to prove that the trivial solution $u=0$ is
an attractor of the semiflow $\Phi_{\lambda}$ at the critical value
$\lambda=\lambda_0$. Generally speaking, it is not easy to check
that the trivial solution is an attractor of the system
$\Phi_{\lam_0}$. However in Theorem 3.2, we obtain a more general
attractor bifurcation result, which tells us that it suffices to
check that the trivial solution is isolated, and then we can obtain
the corresponding attractor bifurcation result.

\vs

In order to obtain the
corresponding attractor bifurcation result of \eqref{e4.1}, we calculate the local
center manifold (see e.g. \cite{C1}) of the trivial solution for the system at $\lam=\lam_0$ to check
that the trivial solution $u=0$ is isolated on the local center manifold.

For the mathematical setting, we consider the Hilbert space
$$
  H=\{\dot{L}^2(\Omega):u(x,t)=u(x+\pi,t)\},\hs H_1=
  \{\dot{H}^4(\Omega):u(x,t)=u(x+\pi,t)\},
  $$
where the dot $"\cdot"$ denotes $\int_0^\pi f\mathrm{d}x=0$ for
$f\in L^2$ or $H^4$, and equip $H$ with the usual inner product and
norm denoted by $(\cdot,\cdot),\|\cdot\|$, respectively.

Let $L_\lambda =A-B_\lambda,$ where $A=(I+\Delta)^2$ defined in
$H_1:=D(A)$, $B_\lambda=\lambda I$. Then $L_\lambda$ is a
sectorial operator. We see that the eigenvalues and the
corresponding eigenvectors of $L_\lambda$ are as follows:
\begin{align}\label{e4.2}
  \lambda_k=(1-4k^2)^2-\lambda, \hs
  e_{k1}(x)=\sqrt{\frac{2}{\pi}}\sin(2kx),\hs
  e_{k2}(x)=\sqrt{\frac{2}{\pi}}\cos(2kx)
\end{align}
for $k\geq 1$ associated with the periodic boundary condition:
$$
  u(x,t)=u(x+\pi,t).
 $$

Set $H_2=D(A^\frac{1}{2})$ and $g=-u^3, u\in H_2$. Then
$g:H_2\rightarrow H$ is a locally Lipschitz continuous mapping, and the system
\eqref{e4.1} can be rewritten as
\begin{align}\label{e4.3}
  u_t+L_\lambda u=g(u).
\end{align}
According to \cite{H,SY}, we deduce that for
each $u_0\in H_2$, the system
\eqref{e4.3} has a unique global strong solution $u(t)$ in $H_2$ with $u(0)=u_0$.

\vs

The system \eqref{e4.1} is a gradient system and one can check that
the system has a natural Lyapunov function $V(u),$
$$
  V(u)=\frac{1}{2}\int_\Omega|(I+\Delta)u|^2\mathrm{d}x-\int_\Omega
  F_\lambda(u)\mathrm{d}x, \Hs {\rm where}\hs
  F_\lambda(s)=\frac{\lambda}{2}s^2-\frac{1}{4}s^4.
  $$
Now we consider the case that $\lambda_0:=\lambda=9$. Then the first
eigenvalue of $L_\lambda$ is $\lambda_1=0$, which has the multiplicity
$2$ and its two eigenvectors are
$$
  q_1(x)=\sin 2x,\hs q_2(x)=\cos 2x.
  $$
Let $E_1$ be the eigenspace spanned by $q_1,q_2$, that is
$$
  E_1={\rm span}\{q_1,q_2\}
  $$
and $E_2=E^\bot_1$. Then $H=E_1\oplus E_2$. The projection $P:
H\rightarrow E_1$ is defined by
$$
  P(w_1+w_2)=(\tilde{w}_1+\tilde{w}_2)\sin2x+(\bar{w}_1+\bar{w}_2)\cos2x,
  $$
where
$$
  \tilde{w}_j=\frac{2}{\pi}\int_0^\pi w_j\sin2x \mathrm{d}x,\hs
  \bar{w}_j=\frac{2}{\pi}\int_0^\pi w_j\cos2x \mathrm{d}x, \Hs
  j=1,2.
  $$
Let $u=u_1+u_2, u_1\in E_1, u_2\in E_2$ and
$u_1=s_1\sin2x+s_2\cos2x, s_1,s_2\in \mathbb{R}$. Then we can rewrite
\eqref{e4.3} in the form
\begin{align}\label{e4.4}
  \left\{\ba{ll}\dot{u}_1
=&
  Pg(u_1+u_2),\\[1ex]
  \dot{u}_2
=&
  -(I-P)L_{\lambda_0} u_2+(I-P)g(u_1+u_2).\ea\right.
\end{align}
Denote $\Phi_\lambda$ the semiflow generated by \eqref{e4.1} and $\Phi^1_\lambda$ the restriction of
$\Phi_\lambda$ on $E_1$. Then we have the following result.

\begin{lem}
There exist positive constants $\beta,\varepsilon$ such that the
assumptions $(\mathbf{A}1)$-$(\mathbf{A}4)$ hold and when
$\lambda\in(\lambda_0-\varepsilon,\lambda_0)$, the trivial solution
$u=0$ of system \eqref{e4.1} is a global attractor of
$\Phi_\lambda$.
\end{lem}
\begin{proof}
It is easy to see that there exist $\beta,\varepsilon>0$ such that assumptions
$(\mathbf{A}1)$-$(\mathbf{A}4)$ hold. The argument of the lemma is standard, we omit the details. One can also see \cite{PR,Y} for details.
\end{proof}


Now we state and prove our main results on attractor bifurcation of \eqref{e4.1}.
\begin{thm}
The trivial solution $u=0$ is isolated for the system $\Phi_{\lambda_0}$ generated by \eqref{e4.1} at $\lam=\lam_0$.
 Then there exist a
closed neighborhood $W$ of $0$ and a one-sided neighborhood
$I^+=(\lambda_0,\lambda_0+\varepsilon_1]$ such that for each $\lam\in
I^+$, $\Phi_\lam$ has an attractor $K_\lam$ with $ K_\lam$ in
$W\setminus \{0\}$ and
$$
  \lim\limits_{\lambda\rightarrow
  \lambda_0}\mathrm{dist}_H(K_{\lambda},0)=0.
$$
Furthermore, $K_\lam$ contains an invariant topological sphere
$\mathbb{S}^{n-1}$.
\end{thm}

\begin{proof}

According to Lemma 3.1, we deduce that there exists a neighborhood
$U_1\subset E_1$ of $0$ such that the system \eqref{e4.4} has a
center manifold mapping $u_2=h(u_1):U_1\rightarrow E_2$ with
$h(0)=h'(0)=0$. Thus the equation which determines the asymptotic
behavior of solutions of \eqref{e4.4} is the following two-dimensional
equation \be\label{e4.6} \dot{u}_1 =
  Pg(u_1+h(u_1)),
\ee where $u_1=s_1\sin2x+s_2\cos2x, s_1,s_2\in \mathbb{R}$.
 By Theorem 3.2, one concludes that it suffices to check
that the trivial solution $u_1=0$ of system \eqref{e4.6} is isolated
for the semiflow $\phi_{\lambda_0}$ generated by \eqref{e4.6}.

In what follows
we calculate the local center manifold (see e.g. \cite{C1}) of $u_1=0$ for
$\phi_{\lambda_0}$ in order to show that the trivial solution $u=0$
is isolated.  Note that \be\label{e4.7}
  u_1^3=s_1^3\sin^32x
 +
  s_2^3\cos^32x
 +
  3s_1^2s_2\sin^22x\cos2x
 +
  3s_1s_2^2\sin2x\cos^22x.
\ee So by the definition of $P$, we obtain \be\label{e4.8}
  Pg(u_1)=-Pu_1^3
 =
  -\frac{3}{4}s_1^3\sin2x
 -
  \frac{3}{4}s_2^3\cos2x
 -
  \frac{3}{4}s_1^2s_2\cos2x
 -
  \frac{3}{4}s_1s_2^2\sin2x.
\ee In order to calculate an approximation to $h(u_1),$ we set
\be\label{5.9}
  (\mathcal{M}_1(\psi))(u_1)
 =
  \psi'(u_1)Pg(u_1+\psi(u_1))
 +
  (I-P)L_{\lambda_0}\psi
 -
  (I-P)g(u_1+\psi(u_1)),
\ee where $\psi:E_1\rightarrow E_2$. To apply Theorem 10 in
\cite{C1}, we choose $\psi$ so that
$(\mathcal{M}_1(\psi))(u_1)=\mathrm{O}(u_1^5)$. If
$\psi(u_1)=\mathrm{O}(u_1^3)$, then
$$
  Pg(u_1+\psi(u_1))=Pg(u_1)+\mathrm{O}(u_1^5)
$$
and \be\label{e4.10}
  (\mathcal{M}_1(\psi))(u_1)
 =(I-P)L_{\lambda_0}\psi
 -
  (I-P)g(u_1+\psi(u_1))+\mathrm{O}(u_1^5).
 \ee
It follows from \eqref{e4.7}, \eqref{e4.8} that
\begin{align}\label{e4.11}
 &\hs-(I-P)g(u_1+\psi(u_1))
 =
  (I-P)u_1^3+\mathrm{O}(u_1^5)\nonumber\\
&=
  (-\frac{1}{2}\sin2x\cos4x-\frac{1}{4}\sin2x)s_1^3
 +
  (\frac{1}{2}\cos2x\cos4x-\frac{1}{4}\cos2x)s_2^3\nonumber\\
&+
  (-\frac{3}{2}\cos2x\cos4x+\frac{3}{4}\cos2x)s_1^2s_2
 +
  (\frac{3}{2}\sin2x\cos4x+\frac{3}{4}\sin2x)s_1s_2^2
 + \mathrm{O}(u_1^5).
\end{align}
Let
\begin{align*}
  \psi
 &=
  (\alpha_1\sin2x\cos4x)s_1^3
 +
  (\alpha_2\cos2x\cos4x)s_2^3
 +
  (\alpha_3\cos2x\cos4x)s_1^2s_2\nonumber\\
 &\hs +
  (\alpha_4\sin2x\cos4x)s_1s_2^2.
\end{align*}
By some elementary computations, one can check
$PL_{\lambda_0}\psi=0,$ and so
\begin{align}\label{e4.12}
&\hs (I-P)L_{\lambda_0}\psi
 =
  L_{\lambda_0}\psi\nonumber\\
&=
  (608\alpha_1\sin2x\cos4x+608\alpha_1\cos2x\sin4x)s_1^3
+
  (608\alpha_2\cos2x\cos4x)s_2^3\nonumber\\
&+
  (-608\alpha_2\sin2x\sin4x)s_2^3
+
  (608\alpha_3\cos2x\cos4x-608\alpha_3\sin2x\sin4x)s_1^2s_2\nonumber\\
&+
  (608\alpha_4\sin2x\cos4x+608\alpha_4\cos2x\sin4x)s_1s_2^2.
\end{align}
Then we conclude from \eqref{e4.10}-\eqref{e4.12} that
\begin{align}\label{5.13}
&\hs (\mathcal{M}_1(\psi))(u_1) =L_{\lambda_0}\psi
 -
  (I-P)g(u_1+\psi(u_1))+\mathrm{O}(u_1^5)\nonumber\\
&=(608\alpha_1\sin2x\cos4x+608\alpha_1\cos2x\sin4x)s_1^3
 +
  (608\alpha_2\cos2x\cos4x)s_2^3\nonumber\\
&+
  (-608\alpha_2\sin2x\sin4x)s_2^3
 +
  (608\alpha_3\cos2x\cos4x-608\alpha_3\sin2x\sin4x)s_1^2s_2\nonumber\\
&+
  (608\alpha_4\sin2x\cos4x+608\alpha_4\cos2x\sin4x)s_1s_2^2\nonumber\\
&+
   (-\frac{1}{2}\sin2x\cos4x-\frac{1}{4}\sin2x)s_1^3
 +
  (\frac{1}{2}\cos2x\cos4x-\frac{1}{4}\cos2x)s_2^3\nonumber\\
&+
  (-\frac{3}{2}\cos2x\cos4x+\frac{3}{4}\cos2x)s_1^2s_2
 +
  (\frac{3}{2}\sin2x\cos4x+\frac{3}{4}\sin2x)s_1s_2^2
 + \mathrm{O}(u_1^5).
\end{align}
Therefore, if
$$
  \alpha_1=\frac{1}{2432},\hs \alpha_2=-\frac{1}{2432},\hs
  \alpha_3=\frac{3}{2432},\hs \alpha_4=-\frac{3}{2432}.
$$
Then
\begin{align*}
  \psi
 &=
  (\frac{1}{2432}\sin2x\cos4x)s_1^3
 -
  (\frac{1}{2432}\cos2x\cos4x)s_2^3
 +
  (\frac{3}{2432}\cos2x\cos4x)s_1^2s_2\nonumber\\
 &\hs -
  (\frac{3}{2432}\sin2x\cos4x)s_1s_2^2,
\end{align*}
and
$$
  (\mathcal{M}_1(\psi))(u_1)=\mathrm{O}(u_1^5).
$$
By Theorem 10 in \cite{C1}, we have
\begin{align}\label{e4.14}
  h(u_1)
&=
  \psi(u_1)+\mathrm{O}(u_1^5)\nonumber\\
&=
  (\frac{1}{2432}\sin2x\cos4x)s_1^3
 -
  (\frac{1}{2432}\cos2x\cos4x)s_2^3\nonumber\\
 &\hs+
  (\frac{3}{2432}\cos2x\cos4x)s_1^2s_2
 -
  (\frac{3}{2432}\sin2x\cos4x)s_1s_2^2+\mathrm{O}(u_1^5).
\end{align}
Substituting \eqref{e4.14} into \eqref{e4.6}, we obtain
\begin{align*}
\dot{s}_1
&=-
   \frac{3}{4}s_1^3
  -
   \frac{3}{4}s_1s_2^2
 +
  \frac{3}{4864}s_1^5
 -
   \frac{9}{4864}s_1^3s_2^2
 +
  \mathrm{O}(u_1^7),\\
\dot{s}_2
&=
  -\frac{3}{4}s_2^3
 -
  \frac{3}{4}s_1^2s_2
 +
  \frac{3}{4864}s_2^5
 -
  \frac{9}{4864}s_1^2s_2^3
 +
  \mathrm{O}(u_1^7),
\end{align*}
from which one can conclude that the trivial solution $u_1=0$ is an
isolated equilibrium. Hence the trivial
solution $u=0$ of system \eqref{e4.1} is isolated, see e.g.\cite{BB}.

According to Theorem 3.2, there exist a closed
neighborhood $W$ and $\varepsilon_1>0$ such that the system
\eqref{e4.1} bifurcates from $(0,\lambda_0)$ an attractor
$K_\lambda$ for $\lambda\in (\lambda_0,\lambda_0+\varepsilon_1]$,
where $K_\lambda$ is the maximal compact invariant set of
$\Phi_\lambda$ in $W\setminus\{0\}$ with
$$
  \lim\limits_{\lambda\rightarrow
  \lambda_0}\mathrm{dist}_H(K_{\lambda},0)=0.
$$
Furthermore, $K_\lam$ contains an invariant topological sphere
$\mathbb{S}^{n-1}$. The proof of Theorem 4.1 is complete.
\end{proof}

\newpage
\centerline{\large\bf{ Appendix A:} Isomorphisms Induced by Projections}

\Vs\Vs Let $X_\lam^i,$ $P_\lam^i$ be the same as in Section 3.1.
Since $P^2_\lam=I-P^1_\lam$, the continuity of $P^1_\lam$ implies
that $P_\lam^2$ is continuous in $\lam$ as well.

By $(\mathbf{A}3)$ we can assume $J_0$ is chosen sufficiently small
so that
$$
||P_\lam^i-P_{\lam_0}^i||\leq c<1,\Hs \A\,\lam\in
J_0,\,\,i=1,2.\eqno(A.1)
$$

As before,  we drop the subscript ``$\lam_0$'' and rewrite
$$X^i=X_{\lam_0}^i,\hs  P^i=P_{\lam_0}^i.$$

\noindent{\bf Proposition A1.}\, {\em  For each $i=1,2$, the
restriction  $P^i|_{X^i_\lam}$ of $P^i$ on $X_\lam^i$ is an
isomorphism between $X^i_\lam$ and $X^i$. } \Vs \noindent{\bf
Proof.} To prove Proposition A1,  let us first  verify  that
$P^i|_{X^i_\lam}$ are  one-to-one mappings.

As $P_\lam^2=I-P_\lam^{1}$,  we deduce that
$$
||P_\lam^{2}-P^{2}||=||P_\lam^{1}-P^{1}||\leq c<1.\eqno(A.2)
$$
In what follows  we argue by contradiction and suppose
$P^i|_{X^i_\lam}$ fails to be a  one-to-one mapping for some $i$.
Then there would exist $x_i\in X^i_\lam$ with $x_i\ne0$ such that
$P^i x_i=0$. Further   by (A.1) and (A.2) we see that
$$
||x_i||=||P_\lam^ix_i||=||P_\lam^ix_i-P^ix_i||\leq
c||x_i||<||x_i||,$$ a contradiction\,!

Now we show that $P^i|_{X^i_\lam}$ are isomorphisms. Since
$P^i|_{X^i_\lam}$ are  one-to-one mappings, one immediately
concludes that $P^1|_{X^1_\lam}$ is an isomorphism. So we only need
to consider the case $i=2$.

Let $Q=P^2+P_\lam^{1}$. Then
$$Q|_{X^2_\lam}=P^2|_{X^2_\lam}+P_\lam^{1}|_{X^2_\lam}=P^2|_{X^2_\lam}\,.$$
Because
$$Q=(I-P^{1})+P_\lam^{1}=I-(P^{1}-P_\lam^{1}),
$$
and  $||P^{1}-P_\lam^{1}||<1$, by the basic knowledge in linear
functional analysis, we know that $Q:X\ra X$ is an isomorphism. To
show that $P^2|_{X^2_\lam}$ is an isomorphism, there remains to
check that $Q {X^2_\lam}=X^2$. For this purpose, it suffices to show
that $Q^{-1}X^2\subset X^2_\lam$.

We argue by contradiction and suppose the contrary. There would
exist $u\not\in X^2_\lam$ such that $Qu\in X^2$. Let
$u=x_\lam+x^2_\lam$, where $x_\lam\in X_\lam^{1}$, and $x_\lam^2\in
X_\lam^2$. Then $x_\lam\ne 0$. We observe that
$$
Qu=(P^2+P_\lam^{1})u=P^2u+P_\lam^{1}(x_\lam+x^2_\lam)=x_\lam+P^2u\in
X^2.
$$
Hence $x_\lam\in X^2$. Thereby we have $x_\lam\in X_\lam^{1}\cap
X^2$.
 It follows that $$P_\lam^{1}x_\lam=x_\lam,\hs P^{1}x_\lam=0.$$
 Thus
 $$
 ||x_\lam||=||P_\lam^{1}x_\lam- P^{1}x_\lam||\leq c||x_\lam||<||x_\lam||.
 $$
 This leads to a contradiction and completes the proof of the proposition. $\Box$

\Vs
Now we define for each $\lam\in J_0$ a linear  operators $T_\lam$ on $X$ as follows:
$$
T_\lam u=\Sig_{1\leq j\leq 2}(P^j|_{X_\lam^j}P^j_\lam )\,u,\Hs u\in
X.
$$
It is trivial to check that $T_\lam$ is an isomorphism with $T_{\lam_0}=I$. Clearly $T_\lam$ is continuous in $\lam$, and
$$
    T_\lam X^i_\lam=\Sig_{1\leq j\leq 2}(P^j|_{X_\lam^j}P^j_\lam)\, X^i_\lam=P^i|_{X_\lam^i} X^i_\lam= X^i,\Hs i=1,2.
$$
Thus we have the following conclusion.
\Vs
\noindent{\bf Proposition A2.} {\em Under the
assumptions $(\mathbf{A}1)$-$(\mathbf{A}3)$, there exists  a  family
of isomorphisms $T_\lam$ {\rm ($\lam\in J_0$)} on $X$  depending
continuously on $\lam$ with $T_{\lam_0}=I$, such that $$
    T_\lam X^i_\lam=X^i_{\lam_0}:=X^i,\Hs i=1,2.\eqno(A3)
$$
}

\medskip
\medskip

\end{document}